\documentclass[12pt]{amsart}
\usepackage{amssymb}%
\usepackage{graphicx}%
\usepackage{amsthm}%
\usepackage{amsmath}%
\usepackage{amsfonts} %
\usepackage{latexsym}%
\usepackage{epsfig}%
\usepackage{epic}%
\usepackage{eufrak}%
\usepackage{amscd}%
\usepackage{color}%
\usepackage[linktoc=all, pagebackref, hyperindex]{hyperref}%
\hypersetup{colorlinks,  citecolor=blue,  filecolor=blue,  linkcolor=red,  urlcolor=black}
\theoremstyle{plain}
  \newtheorem{thm}{Theorem}[section]
  \newtheorem{prop}[thm]{Proposition}
  \newtheorem{lem}[thm]{Lemma}
  \newtheorem{cor}[thm]{Corollary}
  
\theoremstyle{definition}
  \newtheorem{dfn}[thm]{Definition}
  \newtheorem{exmp}[thm]{Example}
  \newtheorem{exmps}[thm]{Examples}

\theoremstyle{remark}
  \newtheorem{rem}[thm]{Remark}
    {\begin{list}
        {\noindent\makebox[0mm][r]{\rm(\arabic{enumi})}}
        {\leftmargin=5.5ex \usecounter{enumi}}
    }
    {\end{list}}
\newcommand{\ara}{\operatorname{ara}}
\newcommand{\Ann}{\operatorname{Ann}}

\newcommand{\Ass}{\operatorname{Ass}}

\newcommand{\depth}{\operatorname{depth}}

\newcommand{\V}{\operatorname{V}}
\newcommand{\Ext}{\operatorname{Ext}}

\newcommand{\Hom}{\operatorname{Hom}}
\newcommand{\height}{\operatorname{height}}

\newcommand{\cd}{\operatorname{cd}}

\newcommand{\Max}{\operatorname{Max}}

\newcommand{\rad}{\operatorname{rad}}

\newcommand{\grade}{\operatorname{grade}}
\newcommand{\Spec}{\operatorname{Spec}}
\newcommand{\Supp}{\operatorname{Supp}}

\def\1{{\mathbf 1}}

\def\fa{{\mathfrak a}}
\def\fb{{\mathfrak b}}

\def\fm{{\mathfrak m}}

\def\fp{{\mathfrak p}}

\def\<{{\langle}}
\def\>{{\rangle}}

\definecolor{myc}{cmyk}{0.3,0.5,1,0}
\newcommand{\excise}[1]{}
\begin{document}

\title[Relative Cohen-Macaulay filtered modules]{RELATIVE COHEN-MACAULAY FILTERED MODULES WITH A VIEW TOWARD RELATIVE COHEN-MACAULAY MODULES}
 
\address{Department of Mathematics, Payame Noor University, Tehran, 19395-3697, Iran.} 
\email{zohouri@phd.pnu.ac.ir}

\author[Mast Zohouri, Ahmadi Amoli and Faramarzi]{M. Mast Zohouri, Kh. Ahmadi Amoli$^{*}$ and S. O. Faramarzi}

\address{Department of Mathematics, Payame Noor University, Tehran, 19395-3697, Iran.}
\email{khahmadi@pnu.ac.ir}

\address{Department of Mathematics, Payame Noor University, Tehran, 19395-3697, Iran.}
\email{faramarzi@tabrizu.ac.ir}



\thanks{{\scriptsize 
$*$Corresponding author }}

\begin{abstract}
Let $R$ be a commutative Noetherian ring, $\fa$ a proper ideal of $R$ and $M$ a finite $R$-module. It is shown that, if $(R,\fm)$ is a complete local ring, then under certain conditions $\fa$ contains a regular element on $D_{R}(H^{c}_{\fa}(M))$, where $c=\cd(\fa,M)$. A non-zerodivisor characterization of relative Cohen-Macaulay modules w.r.t $\fa$ is given. We introduce the concept of relative Cohen-Macaulay filtered modules w.r.t $\fa$ and study some basic properties of such modules. In paticular, we provide a non-zerodivisor characterization of relative Cohen-Macaulay filtered modules w.r.t $\fa$. Furthermore, a characterization of cohomological dimension filtration of $M$ by the associated prime ideals of its factors is established. As a consequence, we present a cohomological dimension filtration for those modules whose zero submodule has a primary decomposition. Finally, we bring some new results about relative Cohen-Macaulay modules w.r.t $\fa$. 
\end{abstract}

\maketitle
\vspace{-2ex}%
\setcounter{tocdepth}{1}%

\section{Introduction }
Throughout this paper, let $R$ denote a commutative Noetherian ring with identity and $\fa$ a proper ideal of $R$. For any non-zero $R$-module $M$, the ith local cohomology module of $M$ is defined as 
$$H^{i}_{\fa}(M):=\lim_{\underset{n\geq 1}{\longrightarrow}} \Ext^{i}_{R}(R/{\fa^n},M).$$
$\V(\fa)$ denotes the set of all prime ideals of $R$ containing $\fa$. For an $R$-module $M$, the {\it cohomological dimension} of $M$ with respect to $\fa$ is defined as $\cd(\fa,M):=\sup \{i\in \mathbb{Z}\mid H^{i}_{\fa}(M)\neq 0\}$ which is known that for a local ring $(R,\fm)$ and $\fa=\fm$, this is equal to the dimension of $M$. For unexplained notation and terminology about local cohomology modules, we refer the reader to \cite{BH} and \cite{BSH}. The notion of {\it cohomological dimension filtration} (abbreviated as {\it cd-filtration}) of $M$ introduced by A. Atazadeh and et al. \cite{ASN} which is a generalization of the concept of dimension filtration that is defined by P. Schenzel \cite{Sch} in local case. For any integer $0\leq i\leq \cd(\fa,M)$, let $M_{i}$ denote the largest submodule of $M$ such that $\cd(\fa,M_{i})\leq i$. Because of the maximal condition of a Noetherian $R$-module the submodules $M_{i}$ of $M$ are well-defined. Moreover, it follows that $M_{i-1}\subseteq M_{i}$ for all $1\leq i\leq\cd(\fa,M)$. In the present article, we will use the concept of relative Cohen-Macaulay modules. An $R$-module $M$ is {\it relative Cohen-Macaulay} w.r.t $\fa$ whenever $H^{i}_{\fa}(M)=0$ for all $i\neq\height_{M}(\fa)$. In other words, $M$ is relative Cohen-Macaulay w.r.t $\fa$ if and only if $\grade(\fa,M)=\cd(\fa,M)$ (see \cite{Z}). Notice that this concept has a connection with a notion which has been studied under the title of {\it cohomologically complete intersection ideals} in \cite{HSch}.  It is well-known that $\height_R(\fa)\leq\cd(\fa,R)\leq\ara(\fa)$. The ideal $\fa$ is called a set-theoretic complete intersection ideal whenever $\height_R \fa=\ara(\fa)$. A set-theoretic complete intersection ideal is a cohomologically complete intersection ideal. Recently, relative Cohen-Macaulay modules have been studied also in \cite{HS}.\\
Sharp \cite{SH} and some other authors have shown that a Cohen-Macaulay local ring $R$ admits a canonical module if and only if it is the homomorphic image of a Gorenstein local ring. In particular, if $R$ is a complete Cohen-Macaulay local ring of dimension $n$, then $w_{R}=\Hom_{R}(H^{n}_{\fm}(R),E(R/ \fm))$ is a canonical module of $R$. 
The outline of the paper is as follows.\\
Section 2 is devoted to discuss main topics of this paper. We initiate this section by showing that if $(R,\fm)$ is a complete local ring and $M$ is relative Cohen-Macaulay w.r.t $\fa$ with $\cd(\fa,M)=c>0$, and $\Supp_{R}(H^{c}_{\fa}(M))\subseteq\V(\fm)$, then $\fa$ contains a regular element on $D_{R}(H^{c}_{\fa}(M))$ (see Corollary~\ref{co}). In this horizon, we prove the following theorem (see Theorem~\ref{p3}).

\begin{thm}\label{p0}
Let $(R,\fm)$ be a local ring and let $M$ be a finite $R$-module with $\cd(\fa,M)=c>0$. Assume that 
$\underline{x}=x_{1},\ldots ,x_{n}\in\fa$ is a regular sequence on both $M$ and $D_R(H^{c}_{\fa}(M))$. Then $$\cd(\fa,M/\underline{x}M)=\cd(\fa,M)-n.$$
\end{thm}

As an application of Theorem~\ref{p0}, we bring the next result (see Corollary~\ref{c3}).
\begin{cor}
Let $(R,\fm)$ be a local ring and $M$ be a finite $R$-module with $\cd(\fa,M)=c>0$.
\begin{itemize}
\item[(i)] Let $x\in\fa$ be a regular element on both $M$ and $D_R(H^{c}_{\fa}(M))$. Then $M$ is relative Cohen-Macaulay w.r.t $\fa$ if and only if $M/xM$ is relative Cohen-Macaulay w.r.t $\fa$.
\item[(ii)] Let $\underline{x}=x_{1},\ldots ,x_{n}\in\fa$ be a regular sequence on both $M$ and $D_R(H^{c}_{\fa}(M))$. Then $M$ is relative Cohen-Macaulay w.r.t $\fa$ if and only if $M/ \underline{x}M$ is relative Cohen-Macaulay w.r.t $\fa$.
\end{itemize}
\end{cor}

We define the {\it cohomological deficiency modules} of $M$ as the matlis duality of $H^{i}_{\fa}(M)$ for $i\neq \cd(\fa,M)$, denoted by $K^{i}_{\fa}(M)$ and further {\it cohomological canonical module} of $M$ if $i=\cd(\fa,M)$.
One of the main purposes of this section is to introduce a new class of modules over $R$ called {\it relative Cohen-Macaulay filtered modules} ({\it or relative sequentially Cohen-Macaulay modules}) w.r.t $\fa$, abbreviated by RCMF modules. This is an extension of the concept of {\it Cohen-Macaulay filtered modules} ({\it sequentially Cohen-Macaulay modules}) introduced by P. Schenzel \cite{Sch} for local case. The concept of sequentially Cohen-Macaulay modules was introduced by Stanley \cite{S} for graded modules. It is interesting that any relative Cohen-Macaulay $R$-module w.r.t $\fa$ is an RCMF $R$-module w.r.t $\fa$. But, any RCMF $R$-module w.r.t $\fa$ is not a relative Cohen-Macaulay $R$-module w.r.t $\fa$ necessarily.
We derive permanence properties of RCMF modules. More precisely, basic properties of RCMF modules with respect to non-zerodivisors, localization and completion are discussed. One of the main results of this section is the following theorem (see Theorem~\ref{Th1}). 

\medskip
\begin{thm}
Let $(R,\fm)$ be a local ring and $M$ be a finite $R$-module with the cd-filtration $\mathcal{M}=\{M_{i}\}^{c}_{i=0}$ where $c=\cd(\fa,M)$.
\begin{itemize}
\item[(i)] Let $x\in\fa$ be a regular element on $M$, $D_R(H^{c}_{\fa}(M))$, and $D_R(H^{i}_{\fa}(M_i))$ for all $0\leq i\leq c$. Then $M$ is an RCMF module w.r.t $\fa$ if and only if $M/xM$ is an RCMF module w.r.t $\fa$.
\item[(ii)] Let $\underline{x}=x_{1},\ldots ,x_{n}\in\fa$ be a regular sequence on $M$,  $D_R(H^{c}_{\fa}(M))$, and $D_R(H^{i}_{\fa}(M_i))$ for all $0\leq i\leq c$. Then $M$ is an RCMF module w.r.t $\fa$ if and only if $M/ \underline{x}M$ is an RCMF module w.r.t $\fa$.
\end{itemize}
\end{thm}

\medskip
As another main result of this section, we provide a necessary and sufficient condition for a filtration to be cd-filtration of a module by the associated prime ideals of its factors (see Theorem~\ref{Th4}). 
\begin{thm}
Let $\mathcal{M}=\{M_i\}_{i=0}^{c}$ be a filtration of the finite $R$-module $M$ and $\cd(\fa,M_{0})=0$. The following conditions are equivalent:
\begin{itemize}
\item[(i)] $\Ass(M_{i}/M_{i-1})=\Ass^{i}(M)$ for all $1\leq i\leq c$;
\item[(ii)] $\mathcal{M}$ is the cd-filtration of $M$.
\end{itemize}
\end{thm}

Atazadeh and et al. \cite[Theorem 1.1]{ASN} proved that if a finite $R$-module $M$ has a cd-filtration, then this filtration is uniquely determined by a reduced primary decomposition of the zero submodule in $M$. In the present paper, without such a condition on $M$, we present a cd-filtration for all $R$-modules whose zero submodule has a primary decomposition (see Corollary~\ref{c5}). 
\\
In section 3, we study relative Cohen-Macaulayness in rings and modules. In view of \cite[Proposition 2.3]{M}, it follows that if $(R,\fm)$ is a Cohen-Macaulay ring, then for a maximal Cohen-Macaulay non-zero module $M$, relative Cohen-Macaulayness of the ring $R$ and the module $M$ is equivalent, whenever $\Supp_R(M)=\Spec(R)$. We determine equivalency between two classes of relative Cohen-Macaulay rings and modules, ``multiplication" and ``semidualizing" modules, in Corollaries~\ref{c1} and ~\ref{cc}. Among other things, in Proposition~\ref{prop31}, comparing with \cite[Proposition 5.1]{M}, we show that $R$ is relative Cohen-Macaulay w.r.t $\fa$ if and only if its canonical module is relative Cohen-Macaulay w.r.t $\fa$. 
Finally, as any relative Cohen-Macaulay module w.r.t $\fa$ is an RCMF module w.r.t $\fa$, if $(R,\fm)$ is a relative Cohen-Macaulay local ring w.r.t $\fa$ with $\cd(\fa,R)=c$, then for all $0\leq i\leq c$ the $R$-modules $K^{i}_{\fa}(R)$ are either zero or $i$-cohomological dimensional relative Cohen-Macaulay modules w.r.t $\fa$ (see Proposition~\ref{prop6}).

\section{Cohomological dimension filtration and relative Cohen-Macaulay filtered modules}
In this section, we define and study relative Cohen-Macaulay filtered modules. As the main objective, a characterization for such modules is presented in Theorem~\ref{Th1} and also a characterization of cd-filtration is presented in Theorem~\ref{Th4}. We begin by recalling the definition of cohomological dimension filtration due to Atazadeh and et al. in \cite{ASN} and the concept of relative Cohen-Macaulayness due to Zargar in \cite{Z}.

\begin{dfn}(see \cite{ASN}) 
Let $M$ be a finite $R$-module. The increasing filtration $\mathcal{M}=\{M_{i}\}^{c}_{i=0}$ of submodules of $M$, where $c=\cd(\fa,M)$ is called the {\it cohomological dimension filtration} of $M$ if for all integer $0\leq i\leq c$, $M_{i}$ is the largest submodule of $M$ such that $\cd(\fa,M_{i})\leq i$. 
\end{dfn}

\medskip
\begin{dfn}(see \cite{Z})
A finite $R$-module $M$ is called {\it relative Cohen-Macaulay} w.r.t $\fa$ if there is precisely one non-vanishing local cohomology module w.r.t $\fa$. Clearly, this is the case if and only if $\grade(\fa,M)=\cd(\fa,M)$. 
\end{dfn}

\medskip
These definitions motivate us to introduce the following concept.
\begin{dfn}\label{defin}
Let $M$ be a finite $R$-module and $\mathcal{M}=\{M_{i}\}^{c}_{i=0}$ be the cohomological dimension filtration of submodules of $M$, where $c=\cd(\fa,M)$. $M$ is called a {\it relative Cohen-Macaulay filtered module} ({\it relative sequentially Cohen-Macaulay module}) w.r.t $\fa$, whenever $\mathcal{M}_{i}={M_{i}}/{M_{i-1}}$ is either zero or an $i$-cohomological dimensional relative Cohen-Macaulay module w.r.t $\fa$ for all $1\leq i\leq c$. Let us abbreviate this notion by RCMF.
\end{dfn}

\medskip
Related to the definition of RCMF modules, we state the notion of {\it relative Cohen-Macaulay filtration} w.r.t $\fa$ which will be useful in the process.
\begin{dfn}
Let $M$ be a fininte $R$-module with $\cd(\fa,M)=c$. An increasing filtration ${\mathcal{C}}=\{C_{i}\}^{c}_{i=0}$ of submodules of $M$ is called a {\it relative Cohen-Macaulay filtration} of $M$ w.r.t $\fa$ whenever $C_{c}=M$ and ${\mathcal{C}_{i}=C_{i}/C_{i-1}}$ is either zero or an $i$-cohomological dimensional relative Cohen-Macaulay module w.r.t $\fa$ for all ${1\leq i\leq c}$.
\end{dfn}
\medskip
\begin{prop}\label{prop1}
Let ${\mathcal{C}}=\{C_{i}\}^{c}_{i=0}$ be the relative Cohen-Macaulay filtration of the $R$-module $M$ w.r.t $\fa$. Then ${\mathcal{C}}$ coinsides with the cohomological dimension filtration. 
\end{prop}
\begin{proof} 
First, it is clear that $\cd(\fa,C_i)\leq i$ for all $0\leq i\leq c$. Also we have $$\Ass_R C_i=\{\fp\in\Ass_R M\mid \cd(\fa,R/\fp)\leq i\}$$ by \cite[Proposition 2.6]{ASN}. This implies that  $C_i=H^{0}_{\fa_i}(C_j)$ in which $\fa_i=\prod_{\cd(\fa,R/ \fp_{j})\leq i}\fp_{j}$ for all $0\leq i\leq c$. Now let $0\leq i\leq c$ and $j\geq i$. Consider the following exact sequence
$$0\longrightarrow C_j\longrightarrow C_{j+1}\longrightarrow \mathcal{C}_{j+1}\longrightarrow 0.$$
As $\mathcal{C}_{j+1}$ is either zero or $(j+1)$-cohomological dimensional relative Cohen-Macaulay module w.r.t $\fa$, it follows that $H^{0}_{\fa_i}(C_j)\cong H^{0}_{\fa_i}(C_{j+1})$. Thus $C_i=H^{0}_{\fa_i}(C_j)$ for all $j\geq i$. Also since $M=C_c$, the proof will be completed by \cite[Proposition 2.3]{ASN}.
\end{proof} 

\begin{rem}\label{R2}
Let $\mathcal{M}=\{M_{i}\}^{c}_{i=0}$ be the cd-filtration of $M$ where $c=\cd(\fa,M)$. Let $1\leq i\leq c$. Considering the exact sequence $0\longrightarrow M_{i-1}\longrightarrow M_{i}\longrightarrow M_{i}/M_{i-1}\longrightarrow 0$, we have 
$$\cd(\fa,M_{i})=\max\{\cd(\fa,M_{i-1}), \cd(\fa,M_{i}/M_{i-1})\}$$
by \cite[Corollary 2.3 (i)]{DNT}. Thus $\cd(\fa,M_{i})=\cd(\fa,M_{i}/M_{i-1})$ for all $1\leq i\leq c$.
\end{rem}

\medskip
Now recall that the {\it finiteness dimension} of $M$ relative to $\fa$, $f_{\fa}(M)$, is defined by 
$$
f_{\fa}(M):=\inf\{i\in\mathbb{N}_{0}: H^{i}_{\fa}(M)\ \text{is not finitely generated}\}.
$$
Another formulation is the {\it $\fb$-finiteness dimension} of $M$ relative to $\fa$ which is defined by 
$$
f_{\fa}^{\fb}(M):=\inf\{i\in\mathbb{N}_{0}: \fb\nsubseteq\rad (0:_{R}H^{i}_{\fa}(M))\}.
$$
Moreover, {\it the $\fb$-minimum $\fa$-adjusted depth} of $M$, denoted by $\lambda_{\fa}^{\fb}(M)$, is defined by 
$$
\lambda_{\fa}^{\fb}(M):=\inf\{\depth M_{\fp}+\height (\fa+\fp)/\fp: \fp\in\Spec(R)\setminus\V(\fb)\},
$$
where $\fb$ is the second ideal of $R$ without assuming $\fb\subseteq\fa$ in general. By convention, the infimum of the empty set of integers is interpreted by $\infty$.
\medskip
\begin{lem}\label{r3}
Let $M$ be a relative Cohen-Macaulay $R$-module w.r.t $\fa$ with $\cd(\fa,M)>0$. Then
$$
\cd(\fa,M)=\cd(\fa,R/\fp)\ \text{for all}\ \fp\in\Ass_{R}(M).
$$
\end{lem}
\begin{proof}
First, note that we have $f_{\fa}(M)=\cd(\fa,M)$ as $M$ is relative Cohen-Macaulay w.r.t $\fa$ and so $\cd(\fa,M)\leq \lambda_{\fa}^{\fa}(M)$ by \cite[Theorem 9.3.7]{BSH}. Let $\fp\in\Ass_R(M)$. If $\fp\in\Ass_{R}(M)\cap\V(\fa)$, then $$\cd(\fa,M)\leq\lambda_{\fa}^{\fa}(M)\leq\depth M_{\fp}+\height(\fa+\fp)/\fp.$$
That is $\cd(\fa,M)\leq 0$ which is a contradiction. Thus $\fp\in\Ass_{R}(M)\setminus \V(\fa)$. Now, by the Independence Theorem \cite[Theorem 4.2.1]{BSH}, we have 
\[\begin{array}{ll}
\cd(\fa,M)\leq\lambda_{\fa}^{\fa}(M)&\leq\depth M_{\fp}+\height(\fa+\fp)/ \fp\\&=\height(\fa+\fp)/ \fp\\&\leq
\cd((\fa+\fp)/ \fp,R/ \fp)\\&=\cd(\fa,R/ \fp)\\&\leq\cd(\fa,M).
\end{array}\]
Therefore $\cd(\fa,M)=\cd(\fa,R/\fp)$.
\end{proof}

\medskip
Here several examples of RCMF modules are provided. In order to prove part (e), we bring the following remark.

\begin{rem}\label{r8}
Let $R\longrightarrow R'$ be a faithfully flat ring homomorphism and $M$ be an $R$-module. Then by using the Flate Base Change Theorem \cite[Theorem 4.3.2]{BSH} we get $\cd(\fa,M)=\cd(\fa R',M\otimes_R R')$ and $\grade(\fa,M)=\grade(\fa R',M\otimes_R R')$.
\end{rem}

\medskip
\begin{exmps}\label{E1}
\begin{itemize}
\item[(a)] Any relative Cohen-Macaulay module w.r.t $\fa$ is an RCMF module w.r.t $\fa$. 
\item[(b)] Let $M$ be an $R$-module with $\cd(\fa,M)=1$. Then $M$ is an RCMF module w.r.t $\fa$.
\item[(c)] Let $R=k[X,Y]/(XY)$ and $\fa=(x)$ be an ideal of $R$ in which $x$ is the image of $X$ in $R$. Then $R$ is RCMF w.r.t $\fa$.
\item[(d)] Let $R$ be a ring with $\cd(\fa,R)=c$. Let $N_{i}$, $0\leq i\leq c$ be a family of $R$-modules such that either $N_{i}=0$ or $N_{i}$ is $i$-cohomological dimensional relative Cohen-Macaulay module w.r.t $\fa$. Then $M=\oplus_{i=0}^{c}N_{i}$ is an RCMF module w.r.t $\fa$ over $R$. 
\item[(e)] Let $R[[x]]$ be the formal power series ring in one variable $x$ over the ring $R$. Then a finite $R$-module $M$ is an RCMF module w.r.t $\fa$ if and only if $M[[x]]$ is an RCMF module over the ring $R[[x]]$ w.r.t $\fa R[[x]]$.
\end{itemize}
\end{exmps}
\begin{proof}
(a) Let $M$ be a relative Cohen-Macaulay $R$-module w.r.t $\fa$ with $\cd(\fa,M)=c$. We can assume that $c>0$. Since $\cd(\fa,M)=\cd(\fa,R/\fp)$ for all $\fp\in\Ass_{R}(M)$ by Lemma~\ref{r3}, we have $M_c=M$, and $M_{i}=0$ for all $0\leq i<c$. It follows that $M$ is an RCMF module w.r.t $\fa$.\\
To prove part (b), let $M_0 \subseteq M_1=M$ be the cd-filtration of $M$. We show that $M/M_0$ is relative Cohen-Macaulay w.r.t $\fa$. By Remark~\ref{R2}, $\cd(\fa,M/M_0)=\cd(\fa,M)=1$. On the other hand, since $\cd(\fa,\Gamma_{\fa}(M))\leq 0$, we have $\Gamma_{\fa}(M_0)=\Gamma_{\fa}(M)$ as $M_0$ is the largest submodule of $M$ with $\cd(\fa,M_0)\leq 0$. Therefore considering the following exact sequence
$$
0\longrightarrow H^{0}_{\fa}(M_0)\longrightarrow H^{0}_{\fa}(M)\longrightarrow H^{0}_{\fa}(M/M_0)\longrightarrow H^{1}_{\fa}(M_0)\longrightarrow\cdots,
$$
we get $H^{0}_{\fa}(M/M_0)=0$ and so $\grade(\fa,M/M_0)=1$ as desired.
\\
For part (c), as it is seen in \cite[Example 2.3]{V}, $\cd(\fa,R)=1\neq 0=\height_R(\fa)$. Hence $R$ is not relative Cohen-Macaulay w.r.t $\fa$. On the other hand, $R$ is RCMF w.r.t $\fa$ as (b) holds. This example shows that being RCMF w.r.t $\fa$ does not lead to relative Cohen-Macaulayness w.r.t $\fa$ necessarily.\\
Part (d) follows by Proposition~\ref{prop1} as $M$ admits a filtration $M_{i}=\oplus_{j=0}^{i} N_{j}$ such that $M_{i}/M_{i-1}\cong N_{i}$, is either zero or an $i$-cohomological dimensional relative Cohen-Macaulay module w.r.t $\fa$ for all $1\leq i\leq c$. One can prove this by definitions and the fact that $\cd(\fa,\oplus_{j=0}^{i}N_j)=i$ for all $0\leq i\leq c$. \\
For the last part, apply Remark~\ref{r8} as $R[[x]]$ is isomorphic to the $I$-adic completion of $R[x]$, where $I=(x)$ and $R[x]$ is faithfully flat over $R$.
\end{proof}

\medskip
\begin{dfn}\label{def}
Let $(R,\fm,k)$ be a local ring, $\fa$ an ideal of $R$ and $M$ be a finite $R$-module with $\cd(\fa,M)=c$. For $i \neq c$, the {\it ith cohomological deficiency module} of $M$ is defined by $$K^{i}_{\fa}(M):=D_{R}(H^{i}_{\fa}(M))=\Hom_{R}(H^{i}_{\fa}(M),E_{R}(k)).$$
The module $K(M):=K^{c}_{\fa}(M)$ is called the {\it cohomological canonical module} of $M$. Note that $K^{i}_{\fa}(M)=0$ for all $i<0$ or $i>c$.
\end{dfn}

\medskip
According to the above definition, we bring the following property of the cohomological canonical module of $R$.
\begin{lem} \label{lemm1}
Let $(R,\fm)$ be a relative Cohen-Macaulay local ring w.r.t $\fa$ with $\cd(\fa,R)=c$. Then $K^{c}_{\fa}(R)$ is a faithful relative Cohen-Macaulay $R$-module w.r.t $\fa$ of finite injective dimension of type one.
\end{lem}
\begin{proof}
Firstly, in \cite[Theorem 4.3 (ii), (iii)]{Z}, it is shown that $K^{c}_{\fa}(R)$ is relative Cohen-Macaulay w.r.t $\fa$ of finite injective dimension of type one. Next, as $R$ is a relative Cohen-Macaulay ring w.r.t $\fa$, we have $\Ann_{R}H^{c}_{\fa}(R)=0$ by \cite[Theorem 3.3]{L}. Hence  $\Ann_{R}K^{c}_{\fa}(R)=0$ because the annihilators of an $R$-module and its Matlis dual are equal by \cite[Remarks 10.2.2]{BSH}. This completes the proof.
\end{proof} 

\medskip
In the next proposition, we provide a cohomological result of RCMF modules.
\begin{prop}\label{lemma}
Let $M$ be an RCMF module w.r.t $\fa$ with ${\mathcal{M}}=\{M_{i}\}^{c}_{i=0}$ its cd-filtration, where $c=\cd(\fa,M)$. Then for all $1\leq i\leq c$, $$H^{i}_{\fa}(M)\cong H^{i}_{\fa}(M_{i})\cong H^{i}_{\fa}({\mathcal{M}}_{i}),$$ where $\mathcal{M}_i=M_i/M_{i-1}$ for all $1\leq i\leq c$. In particular, it follows that $K^{i}_{\fa}(M)\cong K^{i}_{\fa}({\mathcal{M}}_{i})$ for all ${1\leq i\leq c}.$
\end{prop}
\begin{proof} 
Let $1\leq i\leq c$. By using the short exact sequence $$0\longrightarrow M_{i-1}\longrightarrow M_{i}\longrightarrow {\mathcal{M}}_{i}\longrightarrow 0$$ and since $\cd(\fa,M_{i-1})\leq i-1$, we have an isomorphism $H^{i}_{\fa}(M_{i})\cong H^{i}_{\fa}({\mathcal{M}}_{i})$. Now, as ${\mathcal{M}}_{i}=M_{i}/M_{i-1}$ is relative Cohen-Macaulay w.r.t $\fa$, it yields isomorphisms $H^{j}_{\fa}(M_{i})\cong H^{j}_{\fa}(M_{i-1})$ for all $j<i$. By induction, it follows that $$H^{i}_{\fa}(M)=H^{i}_{\fa}(M_{c})\cong H^{i}_{\fa}(M_{c-1})\cong \ldots \cong H^{i}_{\fa}(M_{i+1})\cong H^{i}_{\fa}(M_{i}),$$ which completes the first part of the assertion. Now by virtue of Definition~\ref{def}, we obtain $K^{i}_{\fa}(M)\cong K^{i}_{\fa}({\mathcal{M}}_{i})$ for all $1\leq i\leq c$ as desired.
\end{proof}

We will provide the definitions and results which are needed in the process. The following definition is a generalization of the concept of $\fa$-filter regular $M$-sequences which has been stated in \cite{CSCHT}. 
\begin{dfn}(see \cite{T})
A sequence $x_1,\ldots,x_n$ of elements of $R$ is called an {\it $\fa$-filter regular $M$-sequence} if $x_i\notin\fp$ for all 
$$
\fp\in\Ass_R(M/(x_1,\ldots,x_{i-1})M)\setminus\V(\fa)\ \text{for all} \ 1\leq i\leq n.
$$
\end{dfn}

\medskip
By definition, it deduces that every regular $M$-sequence is an $\fa$-filter regular $M$-sequence and any $R$-filter regular $M$-sequence is a poor regular $M$-sequence.
\begin{rem}\label{re}
\begin{itemize}
\item[(i)] Let $n$ be a positive integer. By definition, we can find $n$ elements of $\fa$ which form an $\fa$-filter regular $M$-sequence as follows. If $\Ass_R(M)\subseteq\V(\fa)$, then choose $y_1\in\fa$ arbitrarily. If not, since $\fa\nsubseteq\cup_{\fp\in\Ass_R(M)\setminus\V(\fa)}\fp$, there exists $y_1\in\fa$ such that $y_1\notin\fp$ for all $\fp\in\Ass_R(M)\setminus\V(\fa)$. Again, if $\Ass_R(M/y_1 M)\subseteq\V(\fa)$, then choose $y_2\in\fa$ arbitrarily. If not, since $\fa \nsubseteq\cup_{\fp\in\Ass_R(M/y_1 M)\setminus\V(\fa)}\fp$, there esists $y_2\in\fa$ such that $y_2\notin\fp$ for all $\fp\in\Ass_R(M/y_1 M)\setminus\V(\fa)$. Proceeding the same way, we can find $y_1,\ldots,y_n\in\fa$ which form an $\fa$-filter regular $M$-sequence.
\item[(ii)] For any positive integer $n$, there are $n$ elements of $R$ which form a poor regular $M$-sequence.
\end{itemize}
\end{rem}

\medskip
\begin{dfn}
(see \cite{O} and \cite{R}) For a ring $R$, let $E_{R}$ be the injective hull of the direct sum $\oplus_{\fm\in\Max(R)}R/\fm$ of all simple $R$-modules and $D_{R}(-)$ be the functor $\Hom_{R}(-,E_{R})$. (Note that $D_{R}(-)$ is a natural generalization of Matlis duality functor to non-local rings.)
\end{dfn}

\medskip 
\begin{lem}\label{lem} 
Let $M$ be a finite dimensional finite $R$-module and $n$ be a positive integer such that $\Ext^{j}_{R}(R/\fa,D_{R}(H^{t}_{\fa}(M)))=0$ for all $t>n$ and all $j\in \mathbb{N}_{0}$. Then for any $\fa$-filter regular $M$-sequence $x_{1},\ldots ,x_{i}$ with $i>n$, $\Ext^{j}_{R}(R/\fa,D_{R}(H^{i}_{(x_{1},\ldots ,x_{i})}(M)))=0$ for all $j\in \mathbb{N}_{0}$. In particular, it holds for $n=\cd(\fa,M)$.
\end{lem}
\begin{proof}
In the proof of \cite[Lemma 3.3]{KhKh}, it is used the fact that for every exact sequence of $R$-modules, the finiteness of sided modules lead to the finiteness of the middle one. Thus the same method of the proof works exactly replacing ``finite modules" by ``zero modules". 
\end{proof}
\medskip
Applying Lemma~\ref{lem}, we obtain the following result which is needed for Corollary~\ref{co}. 
\begin{lem}\label{lemm}
Let $n$ be a positive integer such that $\Ext^{j}_{R}(R/\fa,D_{R}(H^{t}_{\fa}(M)))=0$ for all $t>n$ and all $j\in \mathbb{N}_{0}$. Then $\Hom_{R}(R/\fa,D_{R}(H^{n}_{\fa}(M)))=0$. In particular, it holds for $n=\cd(\fa,M)$.
\end{lem}
\begin{proof}
By Grothendieck's Vanishing Theorem \cite[Theorem 6.1.2]{BSH}, we may assume that $n\leq\dim M$. In view of Remark~\ref{re}(i) and using Lemma~\ref{lem}, we can trace the same method of the proof of \cite[Theorem 3.4]{KhKh} by replacing ``finite modules" by ``zero modules" to get the assertion.
\end{proof}

Notice that the assumption ``complete local" in \cite[Lemma 3.3 and Theorem 3.4]{KhKh} is due to show that the finiteness properties of local cohomology modules. But we eliminated it in Lemma~\ref{lemm} because we do not need this assumption for vanishing of local cohomology modules.

\medskip
\begin{dfn}
(see \cite{H}) An $R$-module $M$ is called $\fa$-{\it cofinite} if $\Supp_{R}(M)\subseteq\V(\fa)$ and $\Ext^{j}_{R}(R/\fa,M)$ is finite for all $j\geq 0$.
\end{dfn}

\medskip
\begin{rem}\label{r2} (see \cite[Theorem 2.1]{DY})
For a finite $R$-module $M$ and a non-negative integer $c$ if $H^{i}_{\fa}(M)$ is $\fa$-cofinite for all $i<c$, then $\Hom_{R}(R/\fa,H^{c}_{\fa}(M))$ is finite.
\end{rem}

\medskip
\begin{lem}\label{l3}
Let $(R,\fm)$ be a complete local ring, $c$ a non-negative integer. Let $H^{i}_{\fa}(M)=0$ for all $i<c$ and $\Supp_{R}(M/\fa M)\subseteq\V(\fm)$. If $\Hom_{R}(R/\fa,D_{R}(H^{c}_{\fa}(M)))=0$, then $\fa$ contains a regular element on $D_{R}(H^{c}_{\fa}(M))$. 
\end{lem}
\begin{proof}
By assumption, $H^{i}_{\fa}(M)$ is $\fa$-cofinite for all $i<c$. Thus $\Hom_{R}(R/\fa,H^{c}_{\fa}(M))$ is finite by Remark~\ref{r2}. As $\Supp_{R}(M/\fa M)\subseteq\V(\fm)$, then $\Supp_{R}(\Hom_{R}(R/\fa,H^{c}_{\fa}(M)))\subseteq\V(\fm)$ and so $\Hom_{R}(R/\fa,H^{c}_{\fa}(M))$ is Artinian. Thus $H^{c}_{\fa}(M)$ is Artinian by \cite[Theorem 7.1.2]{BSH}. Now, as $R$ is a complete local ring, $D_{R}(H^{c}_{\fa}(M))$ is finite by \cite[Theorem 10.2.12]{BSH}. Also, since $\Hom_{R}(R/\fa,D_{R}(H^{c}_{\fa}(M)))=0$ by our assumption, $\fa$ contains a regular element on $D_{R}(H^{c}_{\fa}(M))$ as required.
\end{proof}

\medskip
Motivated by \cite[Lemma 4.3]{HSch}, we are particularly interested to prove the existance of a regular element on $D_{R}(H^{c}_{\fa}(M))$ as follows.
\begin{cor}\label{co}
Let $(R,\fm)$ be a complete local ring and $M$ be a relative Cohen-Macaulay $R$-module w.r.t $\fa$ with $\cd(\fa,M)=c>0$ and $\Supp_{R}(M/\fa M)\subseteq\V(\fm)$. Then $\fa$ contains a regular element on $D_{R}(H^{c}_{\fa}(M))$. 
\end{cor}
\begin{proof}
It is straightforward by using Lemma~\ref{lemm} and Lemma~\ref{l3}.
\end{proof}

\medskip
Now, we consider the behavior of cohomological dimension of an $R$-module under non-zerodivisors.
\begin{thm}\label{p3}
Let $(R,\fm)$ be a local ring and $M$ be a finite $R$-module with $\cd(\fa,M)=c>0$. If $\underline{x}=x_{1},\ldots ,x_{n}\in\fa$ is a regular sequence on both $M$ and $D_R(H^{c}_{\fa}(M))$, then 
$$
\cd(\fa,M/\underline{x}M)=\cd(\fa,M)-n.
$$
\end{thm}
\begin{proof}
We prove by induction on the length of regular sequence $\underline{x}$. Let $n=1$ and $\widehat{R}$ be the $\fm$-adic completion of $R$. In view of the Flat Base Change Theorem \cite[Theorem 4.3.2]{BSH} as $\widehat{R}$ is faithfully flat over $R$, we have $\cd_R(\fa,M)=\cd_{\widehat{R}}(\fa\widehat{R},\widehat{M})$ and $\cd_R(\fa,M/x_1 M)=\cd_{\widehat{R}}(\fa\widehat{R},\widehat{M}/x_1\widehat{M})$. Thus we may assume that $R$ is a complete local ring. Let $\cd(\fa,M)=c$. It is clear that $\cd(\fa,M/x_1 M)\leq c$. We claim that $c-1\leq\cd(\fa,M/x_1 M)$. Suppose that $\cd(\fa,M/x_1 M)<c-1$. Then the multiplication map by $x_1$ on $H^{c}_{\fa}(M)$ is injective. As $H^{c}_{\fa}(M)$ is $\fa$-torsion, we get $H^{c}_{\fa}(M)=0$ which is a contradiction. Thus $c-1\leq\cd(\fa,M/x_1 M)\leq c$. Now, we show that $H^{c}_{\fa}(M/x_1 M)=0$. From the exact sequence
$$
H^{c}_{\fa}(M)\overset{x_1}\longrightarrow H^{c}_{\fa}(M)\longrightarrow H^{c}_{\fa}(M/x_1 M)\longrightarrow 0,
$$
we obtain the following exact sequence
$$
D_{R}(H^{c}_{\fa}(M/x_1 M))\longrightarrow D_{R}(H^{c}_{\fa}(M))\overset{x_1}\longrightarrow D_{R}(H^{c}_{\fa}(M))\longrightarrow 0.
$$
Since $x_1$ is a regular element on $D_R(H^{c}_{\fa}(M))$, we get $D_{R}(H^{c}_{\fa}(M/x_1 M))=0$. Since the annihilators of $H^{c}_{\fa}(M/x_1 M)$ and its Matlis dual are equal by \cite[Remark 10.2.2]{BSH}, it deduces that $H^{c}_{\fa}(M/x_1 M)=0$. Therefore $\cd(\fa,M/x_1 M)=c-1$, as desired. Now, assume that the assertion holds for any sequences of length $n-1$. Then the claim will be proved by using inductive hypothesis.
\end{proof}

\medskip
We are now ready to bring an effective application of the above theorem which is needed in Theorem~\ref{Th1}.
\begin{cor}\label{c3} 
Let $(R,\fm)$ be a local ring and $M$ be a finite $R$-module with $\cd(\fa,M)=c>0$.
\begin{itemize}
\item[(i)] Let $x\in\fa$ be a regular element on both $M$ and $D_R(H^{c}_{\fa}(M))$. Then $M$ is relative Cohen-Macaulay w.r.t $\fa$ if and only if $M/xM$ is relative Cohen-Macaulay w.r.t $\fa$.
\item[(ii)]Let $\underline{x}=x_{1},\ldots ,x_{n}\in\fa$ be a regular sequence on both $M$ and $D_R(H^{c}_{\fa}(M))$. Then $M$ is relative Cohen-Macaulay w.r.t $\fa$ if and only if $M/ \underline{x}M$ is relative Cohen-Macaulay w.r.t $\fa$.
\end{itemize}
\end{cor}
\begin{proof}
(i) As $\cd(\fa,M/xM)=\cd(\fa,M)-1$ by Theorem~\ref{p3}, and $\grade(\fa,M/xM)=\grade(\fa,M)-1$, the assertion follows easily.\\
Part (ii) follows by part (i) and using induction on $n$.
\end{proof}
\medskip
We are now in a position to bring a non-zerodivisor characterization of RCMF modules. 
\begin{thm}\label{Th1}
Let $(R,\fm)$ be a local ring and $M$ be a finite $R$-module with the cd-filtration $\mathcal{M}=\{M_{i}\}^{c}_{i=0}$ where $\cd(\fa,M)=c$.
\begin{itemize}
\item[(i)] Let $x\in\fa$ be a regular element on $M$, $D_R(H^{c}_{\fa}(M))$, and $D_R(H^{i}_{\fa}(M_i))$ for all $0\leq i\leq c$. Then $M$ is an RCMF module w.r.t $\fa$ if and only if $M/xM$ is an RCMF module w.r.t $\fa$.
\item[(ii)] Let $\underline{x}=x_{1},\ldots ,x_{n}\in\fa$ be a regular sequence on $M$,  $D_R(H^{c}_{\fa}(M))$, and $D_R(H^{i}_{\fa}(M_i))$ for all $0\leq i\leq c$. Then $M$ is an RCMF module w.r.t $\fa$ if and only if $M/ \underline{x}M$ is an RCMF module w.r.t $\fa$.
\end{itemize}
\end{thm}

\begin{proof}
For part (i), let $\mathcal{M}_{i}=M_i/M_{i-1}$. First, as $x\in\fa$ is an $M$-regular element, we show that $M_{0}=0$. By the notion of \cite[Proposition 2.3]{ASN}, we know that $M_{0}=H^{0}_{\fa_{0}}(M)$ where $\fa_{0}=\prod_{\cd(\fa,R/ \fp_{j})=0}\fp_{j}$. Let $\cd(\fa,R/ \fp_{j})=0$. Then we have $H^{0}_{\fa}(R/ \fp_{j})\neq 0$ and so there exists $r\in R\setminus \fp_{j}$ and $n\in\mathbb{N}$ such that $r\fa^{n}\subseteq \fp_{j}$. Thus $\fa\subseteq \fp_{j}$ that is impossible because $\fp_{j}\in\Ass_{R}(M)$. Therefore $\{\fp_{j}\mid \cd(\fa,R/ \fp_{j})=0\}=\emptyset$ and hence $H^{0}_{\fa_{0}}(M)=H^{0}_{R}(M)=0$. That is, $M_{0}=0$ as we claimed. Now, let $1\leq i\leq c$. Then $\Ass_{R}(\mathcal{M}_{i})\subseteq \Ass_{R}(M)$ by \cite[Proposition 2.6]{ASN}. Therefore $x$ is also $\mathcal{M}_{i}$-regular element for all $i\geq 1$. Also, it is easy to see that $M_{i}\cap xM=xM_{i}$ and $xM_i\cap M_{i-1}=xM_{i-1}$ for all $1\leq i \leq c$. (Note that if $m\in M$ and $xm\in M_i$, then $m\in M_i$ because $x$ is a non-zerodivisor on $\mathcal{M}_i$. Hence  $xM\cap M_i\subseteq xM_i$.) Now, let $M$ be an RCMF module w.r.t $\fa$. Then we have 

\[\begin{array}{ll}\label{a1}
\mathcal{M}_{i}/x\mathcal{M}_{i}\cong M_i/(xM_i+M_{i-1}) & \cong (M_i/xM_i)/(M_{i-1}/(xM_i \cap M_{i-1})) \\ & \cong (M_{i}/(M_{i}\cap xM))/(M_{i-1}/(xM\cap M_{i-1})) \\ & \cong ((M_{i},xM)/xM)/((M_{i-1},xM)/xM).\end{array}\]

Therefore as $x$ is $\mathcal{M}_i$-regular element, $\{(M_{i+1},xM)/xM\}^{c-1}_{i=0}$ is a relative Cohen-Macaulay filtration of $M/xM$ w.r.t $\fa$ by Corollary~\ref{c3}. Notice that $H^{i}_{\fa}(M_i)\cong H^{i}_{\fa}(\mathcal{M}_i)$ for all $i$. Therefore by virtue of Proposition~\ref{prop1}, $M/xM$ is an RCMF module w.r.t $\fa$. Moreover, by Theorem~\ref{p3}, $M/xM$ is a $(c-1)$-cohomological dimensional RCMF module w.r.t $\fa$.
Conversely, suppose that $M/xM$ is an RCMF module w.r.t $\fa$. Then the cohomological dimension filtration $\{M'_{i}\}^{c-1}_{i=0}$ of $M/xM$ has the property that $\mathcal {M'}_{i}=M'_{i}/M'_{i-1}$ is either zero or $i$-cohomological dimensional relative Cohen-Macaulay module w.r.t $\fa$. Let $M_{i+1}$ denote the preimage of $M'_{i}$ in $M$,  $\mathcal{M}_{i+1}:=M_{i+1}/M_i$ for $0\leq i\leq c-1$, and $M_{0}:=0$. Since $\Ass_R\mathcal{M'}_i\subseteq\Ass_R M$ by \cite[Proposition 2.6]{ASN}, it deduces that $\Ass_R\mathcal{M}_i\subseteq\Ass_R M$ and so $x$ is an $\mathcal{M}_{i}$-regular element. By the isomorphisms as we mentioned in above, $\mathcal{M}_i/x\mathcal{M}_i\cong M'_{i-1}/M'_{i-2}$ for all $1\leq i\leq c$, as all $M_i$'s contain $xM$ (here put $M'_{-1}:=0$). Thus $\mathcal{M}_i/x\mathcal{M}_i$ is relative Cohen-Macaulay w.r.t $\fa$ and $\cd(\fa,\mathcal{M}_i/x\mathcal{M}_i)=\cd(\fa,M'_{i-1}/M'_{i-2})$ for all $1\leq i\leq c$. Now, by Theorem~\ref{p3} and Corollary~\ref{c3}, $\mathcal{M}_i$ is $i$-cohomological dimensional relative Cohen-Macaulay module w.r.t $\fa$ as $x$ is non-zerodivisor on $\mathcal{M}_i$. Therefore $M$ is RCMF module w.r.t $\fa$.\\ 
Part (ii) follows by part (i) and applying the induction on $n$.
\end{proof}

\medskip
Another property of RCMF modules is about localization behaviour. It is clear that if $M$ is a relative Cohen-Macaulay $R$-module w.r.t $\fa$ and $\fp\in \Supp_{R}(M)\cap\V(\fa)$, then $M_{\fp}$ is relative Cohen-Macaulay $R_{\fp}$-module w.r.t $\fa R_{\fp}$ and $\cd(\fa,M)=\cd(\fa R_{\fp},M_{\fp})$.

\begin{prop}
Let $M$ be an RCMF $R$-module w.r.t $\fa$. Then $M_{\fp}$ is an RCMF $R_{\fp}$-module w.r.t $\fa R_{\fp}$ for any prime ideal $\fp\in \Supp_{R}(M)\cap\V(\fa)$.
\end{prop}
\begin{proof} Let $\mathcal{M}=\{M_{i}\}^{c}_{i=0}$ denote the cd-filtration of $M$, where $c=\cd(\fa,M)$. Let $\fp\in\V(\fa)$. Consider the filtration $\{M_{i}\otimes_{R} R_{\fp}\}^{c}_{i=0}$. We claim that this is a cd-filtration of $M_\fp$. First, as
$$
\cd(\fa,M)=\max\{\cd(\fa,M_{c-1}),\cd(\fa,M/M_{c-1})\},
$$
and
$$
\cd(\fa R_{\fp},M_{\fp})=\max\{\cd(\fa R_{\fp},(M_{c-1})_{\fp}),\cd(\fa R_{\fp},(M/M_{c-1})_{\fp})\},
$$
and $M/M_{c-1}$ is relative Cohen-Macaulay w.r.t $\fa$, we have $\cd(\fa R_{\fp},(M/M_{c-1})_{\fp})=\cd(\fa,M/M_{c-1})=c$ and $\cd(\fa R_{\fp},(M_{c-1})_{\fp})\leq \cd(\fa,M)=c$. Thus $\cd(\fa R_{\fp},M_{\fp})=c$. Further, $M_{i}/M_{i-1}\otimes_{R} R_{\fp}$ is either zero or a relative Cohen-Macaulay $R_{\fp}$-module w.r.t $\fa R_{\fp}$ of cohomological dimension $i$ for all $1\leq i\leq c$. Therefore in view of Proposition~\ref{prop1}, the claim is proved.
\end{proof}

\medskip
We can also show that passage to completion preserves the property of RCMF, as illustrated below.
\begin{prop}
Let $M$ be a finite RCMF $R$-module w.r.t $\fa$. Then $\widehat{M}$ is an RCMF $\widehat{R}$-module w.r.t $\fa\widehat{R}$, where $\widehat{}$ is the $\fa$-adic completion.
\end{prop}
\begin{proof} Let $\{M_{i}\}^{c}_{i=0}$ denote the relative Cohen-Macaulay filtration of $M$, where $c=\cd(\fa,M)$. Then $\{\widehat{M_{i}}\}^{c}_{i=0}$ is a relative Cohen-Macaulay filtration of the $\widehat{R}$-module $\widehat{M}$ w.r.t $\fa\widehat{R}$ by Flate Base Change Theorem \cite[Theorem 4.3.2]{BSH}. Thus by Proposition~\ref{prop1}, $\widehat{M}$ is an RCMF module over $\widehat{R}$ w.r.t $\fa\widehat{R}$.
\end{proof}

\medskip
The other main result of this section is a characterization of the cd-filtration of $M$ in terms of associated prime ideals of its factors. For all $i$, set $\Ass^{i}_R(M)=\{\fp\in\Ass_R(M)\mid \cd(\fa,R/\fp)=i\}$.
\begin{thm}\label{Th4}
Let $\mathcal{M}=\{M_i\}_{i=0}^{c}$ be a filtration of the finite $R$-module $M$ and $\cd(\fa,M_{0})=0$. The following conditions are equivalent:
\begin{itemize}
\item[(i)] $\Ass_R(M_{i}/M_{i-1})=\Ass^{i}_R(M)$ for all $1\leq i\leq c$;
\item[(ii)] $\mathcal{M}$ is the cd-filtration of $M$.
\end{itemize}
\end{thm}
\begin{proof}
By virtue of \cite[Proposition 2.6 (iii)]{ASN}, we only have to prove the implication $(i)\Rightarrow (ii)$. First, we claim that 
$$
\Ass_R(M_{i-1})\cap\Ass_R(M_{i}/M_{i-1})=\emptyset \ \text{for all}\  1\leq i\leq c.
$$
Contrarily, assume that for some $1\leq i\leq c$, there is $\fp\in\Ass_R(M_{i-1})\cap \Ass_R(M_{i}/M_{i-1})$. Then $\cd(\fa,M_{i-1})\geq i$ by (i). If $i>1$, by hypothesis, $\Ass_R(M_{i-1}/M_{i-2})=\Ass^{i-1}_R(M)$ and so $\fp\notin \Ass_R(M_{i-1}/M_{i-2})$. Thus, considering the exact sequence $0\longrightarrow M_{i-2}\longrightarrow M_{i-1}\longrightarrow M_{i-1}/M_{i-2}\longrightarrow 0$, we have $\fp\in \Ass_R(M_{i-2})$. As $\cd(\fa,R/ \fp)=i$, we have $\cd(\fa,M_{i-2})\geq i$. By repeating this descending process, 
$$
\cd(\fa,M_{0}), \cd(\fa,M_{1}), \ldots , \cd(\fa,M_{i-2}), \cd(\fa,M_{i-1})
$$
are all not less than $i$. Hence $\cd(\fa,M_{0})\geq i>0$ which is a contradiction. Now, consider  the exact sequence $0\longrightarrow M_{c-1}\longrightarrow M\longrightarrow M/M_{c-1}\longrightarrow 0$, we have $\cd(\fa,M_{c-1})\leq c-1$ as $\Ass_R(M_{c-1})\cap\Ass_R(M/M_{c-1})=\emptyset$. Now, let $N$ be the largest submodule of $M$ such that $\cd(\fa,N)\leq c-1$ and $\fp\in\Ass_R(N/M_{c-1})$. Since $\Ass_R(N/M_{c-1})\subseteq\Ass_R^{c}(M)$, we have $\cd(\fa,R/\fp)=c$. But  $\fp\in\Supp_R(N/M_{c-1})\subseteq\Supp_R(N)$ and so $\cd(\fa,R/\fp)\leq\cd(\fa,N)\leq c-1$ which is impossible. Therefore $\Ass_R(N/M_{c-1})=\emptyset$ and $M_{c-1}$ is the largest submodule of $M$ such that $\cd(\fa,M_{c-1})\leq c-1$. Now, descendingly, we proceed this method to prove that $\mathcal{M}$ is the cd-filtration of $M$.
\end{proof}

We end this section by a consequence of the above theorem which gives us a cd-filtration for certain modules. 
\begin{cor} \label{c5} (compare \cite[Proposition 2.3]{ASN})
Let $\cap^{n}_{i=1}N_{i}$ be a reduced primary decomposition of $(0)$ in $M$, where $N_{i}$ is $\fp_{i}$-primary, and $M_{i}=\cap_{\cd(\fa,R/\fp_{j})>i} N_{j}$ for all $0\leq i\leq c=\cd(\fa,M)$. Assume that $\cd(\fa,\cap_{\cd(\fa,R/\fp_{j})>0} N_{j})=0$. Then $\{M_{i}\}^{c}_{i=0}$ is the cd-filtration of $M$.
\end{cor}
\begin{proof}
Let $1\leq i\leq c$. It is easy to see that $M_{i-1}=M_i \cap L_i$, where $L_i$ is the intersection of all $N_j$'s such that $\cd(\fa,R/\fp_j)=i$. By rewriting the indices, let $L_i=N_1 \cap\ldots\cap N_m$. By Theorem~\ref{Th4}, we need to show that $\Ass_R(M_i/M_{i-1})=\{\fp_1,\ldots,\fp_m\}$. First, we note that $\Ass_R(M_i/M_{i-1})=\Ass_R(M_i+L_i/L_i)\subseteq\Ass_R(M/L_i)$. Also, $\Ass_R(M/L_i)=\Ass_R(\oplus_{j=1}^{n}M/N_j)=\{\fp_1,\ldots,\fp_m\}$. Thus, it is enough to show that $\{\fp_1,\ldots,\fp_m\}\subseteq\Ass_R(M_i/M_{i-1})$. We have $M_{i-1}=L_i\cap L_{i+1}\cap\ldots\cap L_c$ and $M_i=L_{i+1}\cap L_{i+2}\cap\ldots\cap L_c$. Let $1\leq r\leq m$. We show that ${\fp}_r\in\Ass_R(M_i/M_{i-1})$. As $(0)=\cap^{n}_{j=1}N_{j}$, is a reduced primary decomposition, it deduces that 
\begin{equation}\label{ee}
M_{i-1}\subsetneqq(N_1\cap\ldots\cap\widehat{N_r}\cap\ldots\cap N_m)\cap L_{i+1}\cap\ldots\cap L_c.
\end{equation}
For convenience, we denote the right side of (1) by $A$ for the rest. So there exists $x\in A$ such that $x\notin M_{i-1}$. Notice that $(M_{i-1}:x)=(N_r:x)$. Now, as $N_r$ is ${\fp}_r$-primary,  there exists $t>0$ such that ${{\fp}_r}^t M\subseteq N_r$. Hence ${{\fp}_r}^t x\subseteq M_{i-1}$. Suppose that $s\geq 0$ is the least integer such that ${\fp_r}^{s+1}x\subseteq M_{i-1}$ and ${\fp_r}^s\nsubseteq M_{i-1}$. This implies that there exists $y\in{\fp_r}^s x$ such that $y\notin M_{i-1}$. Now, it is easy to see that $\fp_r=(M_{i-1}:y)$, i.e., $\fp_r\in\Ass_R(M_i/M_{i-1})$. This completes the proof.
\end{proof}

\section{Relative Cohen-Macaulayness in rings and modules}
In this section, we prove some results concerning relative Cohen-Macaulay rings and modules. We begin by determining two classes of modules for which relative Cohen-Macaulayness is equivalent for rings and modules. 

\begin{dfn}(see \cite{ES}) An $R$-module $M$ is called {\it multiplication} if for every submodule $N$ of $M$ there exists an ideal $\fa$ of $R$ such that $N=\fa M$. Moreover, if $\Ann_R M=0$, then $M$ is called faithful multiplication.
\end{dfn}

\medskip
\begin{cor}\label{c1}
Let $M$ be a faithful multiplication $R$-module. Then $M$ is relative Cohen-Macaulay w.r.t $\fa$ if and only if $R$ is relative Cohen-Macaulay w.r.t $\fa$.
\end{cor}
\begin{proof} Using \cite[Theorem 2.2 (b)]{AFM}, we have $\Supp(M)=\Spec(R)$. Thus $\cd(\fa,M)=\cd(\fa,R)$ by \cite[Theorem 2.2]{DNT}. On the other hand, $\grade(\fa,M)=\grade(\fa,R)$ by \cite[Theorem 2.2 (a)]{AFM}. Therefore $\grade(\fa,M)=\cd(\fa,M)$ if and only if $\grade(\fa,R)=\cd(\fa,R)$ as required.
\end{proof}

\medskip
\begin{dfn}(see \cite{F} and \cite{V2})
The $R$-module $M$ is {\it semidualizing} if it satisfies the following:
\begin{itemize}
\item[(i)] $M$ is finitely generated,
\item[(ii)] The homotopy map ${\chi}^{R}_{M}: R\longrightarrow \Hom_{R}(M,M)$, defined by $r\mapsto [s\mapsto rs]$, is an isomorphism, and
\item[(iii)] $\Ext^{i}_{R}(M,M)=0$ for all $ i>0.$
\end{itemize}
\end{dfn}

\medskip
\begin{cor}\label{cc}
Let $M$ be a semidualizing $R$-module. Then $M$ is relative Cohen-Macaulay w.r.t $\fa$ if and only if $R$ is relative Cohen-Macaulay w.r.t $\fa$.
\end{cor} 
\begin{proof}The assertion easily follows by \cite[Proposition 2.1.16]{W} and \cite[Theorem 2.2.6]{W}.
\end{proof}

\medskip
Comparing with \cite[Proposition 5.1]{M}, we prove the following result on canonical modules by getting benefit from the concept of semidualizing modules.
\begin{prop}\label{prop31}
Let $(R,\fm)$ be a local ring and $w_{R}$ be its canonical module. Then $R$ is relative Cohen-Macaulay w.r.t $\fa$ if and only if $w_{R}$ is relative Cohen-Macaulay w.r.t $\fa$.
\end{prop}
\begin{proof}
As the canonical modules are semidualizing, we have $\grade(\fa,w_{R})=\grade(\fa,R)$ by \cite[Theorem 2.2.6]{W}. On the other hand, $\Supp_{R}(w_{R})=\Spec(R)$ by \cite[Proposition 2.1.16]{W} and so $\cd(\fa,w_{R})=\cd(\fa,R)$ by \cite[Theorem 2.2]{DNT}. Therefore $\grade(\fa,R)=\cd(\fa,R)$ if and only if $\grade(\fa,w_{R})=\cd(\fa,w_{R})$, as required.
\end{proof}

\medskip
Next result is about the cohomological deficiency modules of a relative Cohen-Macaulay ring. 
\begin{prop}\label{prop6} 
Let $(R,\fm)$ be a relative Cohen-Macaulay local ring w.r.t $\fa$ with $\cd(\fa,R)=c$. Then for all $0\leq i\leq c$, the $R$-modules $K^{i}_{\fa}(R)$ are either zero or $i$-cohomological dimensional relative Cohen-Macaulay modules w.r.t $\fa$.
\end{prop}
\begin{proof} By assumption and Example~\ref{E1}(a), we deduce that the quotient ideal $I_{i}:={\fa_{i}}/{\fa_{i-1}}$ in the cd-filtration $\{\fa_i\}_{i=0}^{c}$ of $R$ is either zero or an $i$-cohomological dimensional relative Cohen-Macaulay ideal w.r.t $\fa$ for all $1\leq i\leq c$ . In view of Proposition~\ref{lemma}, it follows that $K^{i}_{\fa}(R)\cong K^{i}_{\fa}({I}_{i})$ for all $0\leq i\leq c$. Since $I_{i}$ is either zero or an $i$-cohomological dimensional relative Cohen-Macaulay ideal w.r.t $\fa$, we have $K^{i}_{\fa}(I_{i})$ is either zero or the cohomological canonical module of $I_{i}$. But the cohomological canonical module of $I_{i}$ is relative Cohen-Macaulay module w.r.t $\fa$ by \cite[Theorem 4.3 (i)]{Z}. Thus the assertion follows.
\end{proof}

\medskip
At the end, we give some results about relative Cohen-Macaulay rings under some mild assumptions. 
\begin{prop}\label{ppp}
Let $f:R\longrightarrow R'$ be a faithfully flat homomorphism of Noetherian rings. Then $R$ is relative Cohen-Macaulay w.r.t $\fa$ if and only if $R'$ is relative Cohen-Macauly ring w.r.t $\fa R'$.
\end{prop}
\begin{proof} 
As $H^{i}_{\fa}(R)\otimes_R R'\cong H^{i}_{\fa R'}(R')$ for all $i\geq 0$, we get $\cd(\fa,R)=\cd(\fa R',R')$ and $\grade(\fa,R)=\grade(\fa R',R')$. Therefore the assertion follows.
\end{proof}

\medskip
\begin{cor}
Let $f:R\longrightarrow R'$ be a faithfully flat homomorphism of Noetherian local rings and $R$ be a relative Cohen-Macaulay w.r.t $\fa$. Then $H^{i}_{\fa}(R')\neq 0$ if and only if $(0:_{R}H^{i}_{\fa R'}(R'))=0$.
\end{cor}
\begin{proof}
Apply Proposition~\ref{ppp} and \cite[Theorem 3.3]{L}.
\end{proof}

\medskip
\begin{exmp}
Let $\fa\subseteq Jac(R)$. Then $\widehat{R}$, the $\fa$-adic completion of $R$, is relative Cohen-Macaulay ring w.r.t $\fa\widehat{R}$ if and only if $R$ is relative Cohen-Macauly ring w.r.t $\fa$.
\end{exmp}

\medskip
\begin{exmp}
The polynomial ring $R[x]$ is relative Cohen-Macaulay ring w.r.t $\fa R[x]$ if and only if $R$ is relative Cohen-Macaulay ring w.r.t $\fa$.
\end{exmp}

{\bf Acknowledgements.} 
The authors are grateful to the reviewer for suggesting several improvements to the manuscript. Moreover, the authors would like to express their thanks to Dr. Raheleh Jafari from Kharazmi University for her useful comments.


\begin{thebibliography}
{CD}
\bibitem{ASN}
A. Atazadeh, M. Sedghi, and R. Naghipour, {\it Cohomological dimension filtration 
and annihilators of top local cohomology modules}, Colloquium Mathematicum, {\bf 139} (2015) 25-35. 

\bibitem{AFM}
H. Ansari-Toroghy, F. Farshadifar and M. Mast-Zohouri, {\it Some remarks on multiplication and comultiplication modules}, Inter. Math. Forum {\bf 4} (6) (2009) 287-291.

\bibitem{BH} 
W. Bruns, J. Herzog, {\it Cohen-Macaulay Rings}, Cambridge Studies in Advanced Mathematics, 1998. 

\bibitem{BSH}
M. P. Brodmann and R.Y. Sharp, {\it Local cohomology: an algebraic introduction with geometric applications}, Cambridge University Press, Cambridge, 2013.

\bibitem{CSCHT} 
N. T. Coung, P. Schenzel and N. V. Trung, {\it Verallgemeinerte Cohen-Macaulay-Moduln}, Math. Nachr. {\bf 85} (1978) 57-73.

\bibitem{DNT}
K. Divaani-Aazar, R. Naghipour and M. Tousi, {\it Cohomological dimension of certain algebraic varieties}, Proc. Amer. Math. Soc. {\bf 130} (2002) 3537-3544.

\bibitem{DY} 
M. T. Dibaei and S. Yassemi, {\it Associated primes and cofiniteness of local cohomology modules}, manuscripta math. {\bf 117} (2005) 199-205.

\bibitem{ES}
Z.A. El-Bast and P.F. Smith, {\it Multiplication modules}, Comm. Algebra {\bf 16} (1988) 755-779.

\bibitem{F}
H. B. Foxby, {\it Gorenstein modules and related modules}, Math. Scand. {\bf 31} (1972) 267-284.

\bibitem{H} 
R. Hartshorne, {\it Affine duality and cofiniteness}, Invent. Math. {\bf 9} (1970) 145-164. 

\bibitem{HSch}
M. Hellus and P. Schenzel, {\it On cohomologically complete intersections}, Journal of Algebra, {\bf 320} (2008) 3733-3748.

\bibitem{HS}
M. Hellus and P. Schenzel, {\it Notes on local cohomology and duality}, Journal of Algebra, {\bf 401} (2014) 48-61.

\bibitem{KhKh} 
K. Khashyarmanesh and F. Khosh-Ahang, {\it On the finiteness properties of Matlis duals of local cohomology modules}, Proc. Indian Acad. Sci. Math. Sci. {\bf 118} (2008), no. 2, 197206.

\bibitem{L}
L. R. Lynch, {\it Annihilators of top local cohomology}, Comm. Algebra {\bf 40} (2012) 542-551.

\bibitem{M}
W. Mahmood, {\it On cohomologically complete intersections in Cohen-Macaulay rings}, Math. Reports {\bf 18}(68), 1 (2016), 21-40.

\bibitem{O} 
A. Ooishi, {\it Matlis duality and the width of a module}, Hiroshima Math. J. {\bf 6}(3) (1976) 573-587. 

\bibitem{R} 
A. S. Richardson, {\it Co-localization, co-support and local cohomology}, Rocky Mountain J. Math. {\bf 36}(5) (2006) 1679-1703. 

\bibitem{Z}
M. Rahro Zargar, {\it Some duality and equivalence results}, \href{http://arxiv.org/pdf/1308.3071v2.pdf}{arXiv:1308.3071v2}.

\bibitem{W}
S. Sather-Wagstaff, {\it Semidualizing~modules}, \href{http://www.ndsu.edu/pubweb/~ssatherw/DOCS/sdm.pdf}{\texttt{http://www.ndsu.edu/pubweb/~ssatherw/DOCS/sdm.pdf}}.

\bibitem{Sch}
P. Schenzel, {\it On the dimension filtration and Cohen-Macaulay filtered modules}, in: Commutative Algebra and Algebraic Geometry, Lecture Notes in Pure Appl. Math., vol. 206, Dekker, New York, (1999) 245-264.

\bibitem{SH}
R.Y. Sharp, {\it On Gorenstien modules over a complete Cohen-Macaulay local ring}, Quart. J. Math. Oxford Ser. {\bf 22} (1971), 425-434.

\bibitem{S}
R.P. Stanley, {\it Combinatorics and Commutative Algebra}, Second edition, Birkhäuser Boston, 1996.

\bibitem{T} 
N. V. Trung, {\it Absolutely superficial sequences}, Math. Proc. Camb. Phil. soc., (1983) 35-47.

\bibitem{V}
M. Varbaro, {\it Cohomological and projective dimensions}, Compositio Math. {\bf 149} (2013) 1203-1210.

\bibitem{V2}
W. V. Vasconcelos, {\it Divisor theory in module categories}, North-Holland Math. stud., vol. 14, North HollandPublishing Co., Amsterdam, 1974.


\end{thebibliography}
\end{document}